\documentclass[10pt]{article}
\usepackage{amssymb}
\usepackage[margin=2.5cm]{geometry}
\usepackage{amsmath}
\usepackage{amsthm}
\usepackage{mathrsfs}
\usepackage{subcaption}
\usepackage{listings}
\usepackage{float}
\usepackage{marvosym}
\usepackage{graphicx}

\newtheorem{defn}{Definition}[section]

\newtheorem{prop}[defn]{Proposition}

\newtheorem{theorem}[defn]{Theorem}
\newtheorem{example}[defn]{Example}

\newcommand{\R}{\mathbb{R}}
\newcommand{\E}{\mathbb{E}}
\newcommand{\Prob}{\mathbb{P}}

\makeatletter
\DeclareRobustCommand{\pder}[1]{%
  \@ifnextchar\bgroup{\@pder{#1}}{\@pder{}{#1}}}
\newcommand{\@pder}[2]{\frac{\partial#1}{\partial#2}}
\makeatother

\begin{document}

\title{The frog model with drift on $\R$}

\author{Josh Rosenberg}
\date{}
\maketitle

    \begin{abstract}Consider a Poisson process on $\R$ with intensity $f$ where $0 \leq f(x)<\infty$ for ${x}\geq 0$ and ${f(x)}=0$ for $x<0$.  The ``points" of the process represent sleeping frogs.  In addition, there is one active frog initially located at the origin.  At time ${t}=0$ this frog begins performing Brownian motion with leftward drift $\lambda$ (i.e. its motion is a random process of the form ${B}_{t}-\lambda {t}$).  Any time an active frog arrives at a point where a sleeping frog is residing, the sleeping frog becomes active and begins performing Brownian motion with leftward drift $\lambda$, independently of the motion of all of the other active frogs.  This paper establishes sharp conditions on the intensity function $f$ that determine whether the model is transient (meaning the probability that infinitely many frogs return to the origin is 0), or non-transient (meaning this probability is greater than 0).  A discrete model with $\text{Poiss}(f(n))$ sleeping frogs at positive integer points (and where activated frogs perform biased random walks on $\mathbb{Z}$) is also examined.  In this case as well, we obtain a similar sharp condition on $f$ corresponding to transience of the model.\end{abstract}

    \section{Introduction}
    
    Since its inception, the term frog model has referred to a system of interacting random walks on a rooted graph.  Specifically, it begins with one active frog at the root and sleeping frogs distributed among the non-root vertices, where the number of sleeping frogs at each non-root vertex are independent random variables (not necessarily identically distributed).  The active frog performs a discrete-time random walk on the graph.  Any time an active frog lands on a vertex containing a sleeping frog, the sleeping frog wakes up and begins performing its own discrete-time random walk (independent of those of the other active frogs).  A variety of different versions of the frog model have been looked at including on the infinite d-ary tree $\mathbb{T}_d$ with one sleeping frog per vertex \cite{HJJ2}, on $\mathbb{T}_d$ with i.i.d. Poisson many sleeping frogs per vertex \cite{HJJ1}, and on $\mathbb{Z}^d$ with one sleeping frog per vertex \cite{TW}.  One of the fundamental questions explored in all of these instances has involved asking if the system is recurrent or transient w.r.t. frogs visiting the root.  This is the form of the question addressed in the present work.
    
    The particular version of the frog model that inspired this paper was the frog model with drift on $\mathbb{Z}$.  In this version of the model the activated frogs perform random walks with some positive leftward drift on the integers, and the numbers of sleeping frogs at each vertex (aside from the origin which begins with a single activated frog) are i.i.d. random variables.  In \cite{Gantert} Nina Gantert and Philipp Schmidt establish tight conditions on the distribution function for the number of sleeping frogs per vertex, that determine whether the system is recurrent or transient (i.e. whether the probability the origin is visited by infinitely many frogs is equal to 1 or 0 respectively).  Specifically, they prove that if $\eta$ refers to a nonnegative integer valued random variable that has the same distribution as the number of sleeping frogs at any nonzero vertex, then$$\Prob_{\eta}(\text{the origin is visited i.o.})=\left\{\begin{array}{ll}0 &\text{if }\E[\text{log}^+\eta]<\infty\\1 &\text{if }\E[\text{log}^+\eta]=\infty\end{array}\right.$$(Note that this result does not depend on the particular value of the leftward drift.  It is only required that the leftward drift be positive).
    
    \medskip
    \noindent
    {\bf The frog model with drift on $\R$.} Here the first subject of study will be a continuous analogue of the model looked at by Gantert and Schmidt.  Start with an active frog at the origin that begins performing Brownian motion with leftward drift $\lambda>0$.  The sleeping frogs all reside to the right of the origin according to a Poisson process with intensity $f:[0, \infty)\rightarrow [0, \infty)$.  Any time an active frog hits a sleeping frog, the sleeping frog wakes up and also begins performing Brownian motion with leftward drift $\lambda$, independent of that of the other active frogs (see Figure \ref{fig:Model} for an illustration).  A more formal construction of this process will not be needed, as most of the analysis involves a related and easily constructed birth-death process.  This paper establishes sharp conditions for the Poisson intensity function $f$, distinguishing between transience (meaning the probability that the origin is hit by infinitely many different frogs is 0), and non-transience (meaning this probability is greater than 0).
    
    \bigskip
    
    \begin{figure}[H]
    \centering
    \includegraphics{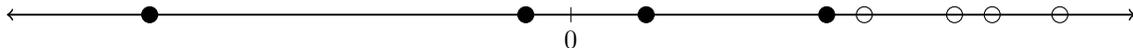}

    \caption{A depiction of the model, where black circles are active frogs \\ and white circles are sleeping frogs.}
    \label{fig:Model}
    
    \end{figure}
    
  \medskip 
    \noindent The main result for the frog model with drift on $\R$ will be the following theorem.  It is assumed here, as well as in the discrete case, that $f$ is not the zero function.
    
  \begin{theorem}\label{theorem:neat3}
  For any $\lambda>0$ and $f$ monotonically increasing, the frog model with drift on $\R$ is transient if and only if\begin{equation}\label{int-cond}\int_0^{\infty}e^{-\frac{f(t)}{2\lambda}}f(t)dt=\infty\end{equation}
  \end{theorem}
  
  \bigskip Before discussing the proof of Theorem \ref{theorem:neat3} we make the following observations about its consequences.  From \eqref{int-cond} it follows that $f(t)=2\lambda\text{log}(t+1)$ represents a critical case with respect to transience vs. non-transience, in the sense that for $f(t)=C\text{log}(t+1)$ a value of $C>2\lambda$ implies non-transience and $C<2\lambda$ implies transience.  More delicate examples include $f(t)=2\lambda\text{log}(t+1)$ (transience), $f(t)=2\lambda\text{log}(t+1)+4\lambda\text{log}\text{log}(t+e)$ (transience), and $f(t)=2\lambda\text{log}(t+1)+(4+\epsilon)\lambda\text{log}\text{log}(t+e)$ for $\epsilon>0$ (non-transience).
  
The proof of Theorem \ref{theorem:neat3} proceeds as follows.  Note first that by virtue of a simple rescaling, it suffices to prove the theorem for the specific case $\lambda=\frac{1}{2}$.  A continuous-time-inhomogeneous birth-death process $\left\{X_t\right\}$ is defined with birth rate $f(t)1_{(X_t>0)}$, and death rate $X_t$.  Transience for the frog model with drift $\frac{1}{2}$ and Poisson intensity of sleeping frogs $f$, is shown to coincide with $X_t$ eventually arriving at the absorbing state 0 with probability 1.  A related process $\left\{Y_t\right\}$ is then defined, which is identical to $\left\{X_t\right\}$ except that 0 is not an absorbing state (i.e. $\left\{Y_t\right\}$ has birth rate $f(t)$ and death rate $Y_t$).  The primary task in proving Theorem \ref{theorem:neat3} consists of proving Theorem \ref{theorem:neat2}, which says that $\left\{Y_t\right\}$ will jump from 0 to 1 infinitely often with probability 1, if and only if the integral expression in \eqref{int-cond} (for $\lambda =\frac{1}{2}$) diverges.  To achieve this, it is first shown that as $t\rightarrow\infty$ the distribution of $Y_t$ behaves increasingly like that of a Poisson r.v. with mean $\int_0^t e^{-(t-u)}f(u)du$.  This is then used to show that the expected number of jumps from 0 to 1 (made by $\left\{Y_t\right\}$) is infinite if and only if the integral expression in \eqref{int-cond} diverges.  Together with a proof that this quantity is infinite with probability 1 as long as it has infinite expectation, this is sufficient for establishing Theorem \ref{theorem:neat2}.  After attending to all of these details in Section \ref{ss:process Yt}, in Section \ref{ss:pt11} it is shown that Theorem \ref{theorem:neat3} follows easily from Theorem \ref{theorem:neat2}.

 \bigskip
 \noindent $Remark\ 1.$ Note that the decision to restrict our focus to the case where no sleeping frogs reside to the left of the origin was not made in order to simplify the problem.  In fact, if the domain of $f$ is expanded to $(-\infty, \infty)$ and it is allowed to take positive values to the left of the origin, then provided$$\int_{-\infty}^0 e^{2\lambda t}f(t)dt<\infty\ ,$$the transience/non-transience of the model depends on the same integral condition from Theorem \ref{theorem:neat3}.  This follows from the theorem, along with the fact that $\E[L_{(-\infty,0)}]$ (where $L_{(a,b)}$ denotes the number of distinct frogs originating in $(a,b)$ that hit the origin) is equal to the above integral.  Alternatively, because $L_{(a,b)}$ (for $b\leq 0$) has a Poisson distribution with mean $\int_a^b e^{2\lambda t}f(t)dt$, divergence of the improper integral above will imply that $L_{(-\infty,0)}$ dominates a Poisson r.v. of any finite mean, thus implying recurrence of the model.   

\medskip
\noindent
{\bf The non-uniform frog model with drift on $\mathbb{Z}$.} After establishing Theorem \ref{theorem:neat3}, we shift our focus towards a discrete model where activated frogs perform simple random walks with leftward drift on $\mathbb{Z}$.  The sleeping frogs are distributed among the positive integer vertices where, for each $j\geq 1$, $\eta_j$ will represent the number of sleeping frogs at $x=j$ at time $t=0$.  The $\eta_j$'s are to be independent Poisson random variables with ${\bf E}[\eta_j]=f(j)$ for some function $f:\mathbb{Z}^+\rightarrow [0,\infty)$.  The process begins with a single active frog at the root and, once activated, frogs perform random walks (independently of each other) which at each step move one unit to the left with probability $p$
(where $\frac{1}{2}<p<1$) and one unit to the right with probability $1-p$.  This model will be referred to as the $\textit{non-uniform frog model with drift on } \mathbb{Z}$ (so as to distinguish it from the model looked at by Gantert and Schmidt).  The terms ``transience" and ``non-transience" will have the same meaning with respect to this model that they had for the model on $\R$.  Once again, a more formal construction will not be necessary since the analysis primarily involves a related discrete-time-inhomogeneous Markov process.  Our main result concerning the model just described will be Theorem \ref{theorem:neat5} (see below), which gives a sharp condition on $f$, distinguishing between transience and non-transience of the model (the condition will also depend on the value of $p$).

Despite their superficial similarities, the non-uniform frog model with drift on $\mathbb{Z}$ and the model looked at by Gantert and Schmidt in \cite{Gantert} are qualitatively quite different.  These differences can be illustrated by noting the contrast between Theorem \ref{theorem:neat5} (below), and the Gantert and Schmidt result discussed earlier.  For the model looked at in \cite{Gantert} in which the $\eta_j$'s are i.i.d. random variables (for all $j\in\mathbb{Z}\backslash\left\{0\right\}$), the sharp condition distinguishing between transience and recurrence involves the exponential tail of the distribution of $\eta_1$.  By contrast, for the non-uniform frog model with drift on $\mathbb{Z}$ in which the $\eta_j$'s are $\text{Poiss}(f(j))$ (for $j\in\mathbb{Z}^+$), the sharp condition distinguishing between transience and (in this case) non-transience involves the asymptotic behavior of $f$.  Note also that unlike with the result in \cite{Gantert}, the particular value of $p$ is taken into account.

Another model which bears a (perhaps stronger) resemblance to the non-uniform model on $\mathbb{Z}$ was examined by Bertacchi, Machado, and Zucca in \cite{BMZ}, where they looked at a frog model on $\mathbb{Z}$ that begins with an active frog at the origin and a single sleeping frog at each positive integer point, and where the frog originating at $i$, upon activation, performs an asymmetric random walk that goes left with probability $p_i$ and right with probability $1-p_i$ (i.e. the drift value can vary depending on the particular frog).  Among the results in their paper, the two that were perhaps the most relevant to my work here, entailed establishing conditions for the individual $p_i$ values (when they are all greater than $\frac{1}{2}$) that guarantee local survival in one case, and local extinction in the other (referred to in this work as non-transience and transience respectively).  The first result (Proposition 2.5 in their paper) states that if there exists a strictly increasing sequence of nonnegative integers $\left\{n_k\right\}$ such that$$\sum_{k=0}^{\infty}\prod_{i=0}^{n_k}\Bigg(1-\Big(\frac{1-p_i}{p_i}\Big)^{n_{k+1}-i}\Bigg)<\infty$$then there is local survival (i.e. non-transience).  For the second result (Proposition 2.10) they show that if $\underset{n\rightarrow\infty}{\text{liminf}}\ p_n>\frac{1}{2}$, then there is local extinction (i.e. transience).  

Perhaps the most pertinent similarity between the model in \cite{BMZ} and the non-uniform frog model with drift on $\mathbb{Z}$ relates to the fact that the long term behavior of both models depends on the asymptotic behavior of sequences ($\left\{p_n\right\}$ in the first case and $f(n)$ in the latter).  A notable difference, however, lies in the fact that the varying parameter in the first case is the drift of individual frogs, and in the second it is the Poisson mean of the distribution of the number of sleeping frogs initially positioned at positive integer points.  Due to this difference, activated frogs in the former case cannot be treated as interchangeable, as they are in the latter.  Consequently, the methods used in \cite{BMZ} are somewhat different from those employed here, and do not result in tight conditions distinguishing between local survival and local extinction being established, such as those we now present for the non-uniform model.

\begin{theorem}\label{theorem:neat5}
For $\frac{1}{2}<p<1$ and $f$ monotonically increasing, the non-uniform frog model with drift on $\mathbb{Z}$ is transient if and only if\begin{equation}\label{sum-inf}\sum_{j=1}^{\infty}e^{-\frac{1-p}{2p-1}f(j)}=\infty\end{equation}
\end{theorem}

\bigskip The proof of Theorem \ref{theorem:neat5} nearly mirrors that of Theorem \ref{theorem:neat3}.  A discrete-time-inhomogeneous Markov process $\{ M_j\}$ is defined, where$$(M_{j+1}|M_j)=\left\{\begin{array}{ll}\text{Bin}\big(M_j,\frac{1-p}{p}\big)+\text{Poiss}\big(\frac{1-p}{p}f(j)\big) &\text{if }M_j\geq 1\\ 0 &\text{if }M_j=0\end{array}\right.$$Transience of the non-uniform frog model is shown to correspond to $\{M_j\}$ eventually arriving at the absorbing state $0$ with probability 1.  $\{N_j\}$ will then represent a process just like $\{M_j\}$, except $(N_{j+1}|N_j=0)=\text{Poiss}(\frac{1-p}{p}f(j))$ (i.e. $0$ is not an absorbing state).  Most of the focus is devoted to proving Theorem \ref{theorem:neat7}, which states that $\{N_j\}$ will attain the value 0 infinitely often with probability 1 if and only if the sum in \eqref{sum-inf} diverges.  The proof involves establishing a series of propositions (\ref{prop:neat8}-\ref{prop:neat10}), which are essentially the discrete analogues of the three propositions (\ref{prop:neat}-\ref{prop:neat4}) that will be used to establish Theorem \ref{theorem:neat2} in the continuous case.  Theorem \ref{theorem:neat5} then follows easily from Theorem \ref{theorem:neat7}.

\bigskip
\noindent $Remark\ 2.$ As with the model on $\R$, allowing sleeping frogs to also reside to the left of the origin does not complicate matters significantly in the case of the model on $\mathbb{Z}$.  If the domain of $f$ is expanded to all of $\mathbb{Z}\backslash\left\{0\right\}$, then the condition given in Theorem \ref{theorem:neat5} continues to apply as long as$$\sum_{j=1}^{\infty}\Big(\frac{1-p}{p}\Big)^j f(-j)<\infty$$since the above sum is equal to $\E[L^*_{(-\infty,0)}]$ (where $L^*_{(a,b)}$ denotes the number of distinct frogs originating in $(a,b)\cap\mathbb{Z}$ that ever hit the origin).  Conversely, since $L^*_{(-N-1,0)}$ has a Poisson distribution with mean equal to the sum of the first $N$ terms in the expression above, the divergence of this sum will, as with the continuous case (see $Remark\ 1$), imply recurrence of the model.

\medskip
After completing the proof of Theorem \ref{theorem:neat5} (along with an accompanying lemma), I conclude the paper with a short section in which I present a pair of counterexamples where $f$ is not monotone increasing, and the tight conditions of Theorems \ref{theorem:neat3} and \ref{theorem:neat5} cease to apply.
  
  \section{Transience vs. non-transience for the model on $\R$}
  
  \subsection{The process $\left\{X_t\right\}$} \label{ss:process Xt}

  Set $\lambda =\frac{1}{2}$.  For any $t\geq 0$ let $X_t$ denote the number of frogs originating in $[0,t]$ that ever pass the point $x=t$.  Now note that $\lim\limits_{t\rightarrow\infty}X_t >0$ if and only if (i) the set of points (initially) containing sleeping frogs is unbounded ($f$ being nonnegative and monotonically increasing implies it is bounded on compact sets, which guarantees that no bounded region will contain infinitely many sleeping frogs) and (ii) all of these frogs are eventually awakened.  Since a Brownian motion with leftward drift $\frac{1}{2}$ is continuous and goes to $-\infty$ with probability 1, it follows that\begin{equation}\label{Xtequiv}\lim\limits_{t\rightarrow\infty}X_t >0\Longleftrightarrow\{\text{infinitely many frogs return to the origin}\}\end{equation}Having established this equivalence, we now present the following proposition.
  
  \begin{prop} \label{prop:bd}
  $\left\{X_t\right\}$ is a continuous-time-inhomogeneous birth-death process with birth rate $f(t)1_{(X_t>0)}$ and death rate $X_t$.
  \end{prop}
  
  \begin{proof}
  By a straight forward argument involving an exponential martingale it is known that the right most point reached by a Brownian motion (beginning at the origin) with leftward drift $\frac{1}{2}$, has an exponential distribution with mean $1$.  By the strong Markov property it follows that if $0<a<b$, then the probability that a frog originating in $[0,a]$ ever passes $x=b$ (conditioned on its passing $x=a$) equals $e^{-(b-a)}$.  Therefore, if we let $R_{(t,dt)}$ represent the number of frogs originating in $[0,t)$ that reach $x=t$, but not $x=t+dt$, then $\big(R_{(t,dt)}|X_t\big)$ has distribution $\text{Bin}(X_t,1-e^{-dt})$.  Since $1-e^{-dt}=dt+o(dt)$ as $dt\rightarrow 0$, it then follows that $\left\{X_t\right\}$ has ``death rate" $X_t$.  Furthermore, since the number of sleeping frogs in $(t,t+dt)$ has distribution $\text{Poiss}(\lambda_{(t,dt)})$ (where $\lambda_{(t,dt)}=\int_t^{t+dt}f(u)du$), and the probability all such frogs are awoken and reach $x=t+dt$, approaches $1$ as $dt\rightarrow 0$ (provided $X_t>0$), this means $\left\{X_t\right\}$ has ``birth rate" $f(t)1_{(X_t>0)}$.  Hence, the proof is complete.
  \end{proof}
  
  \medskip
  \noindent $Remark\ 3.$ While the subscript $t$ denoted a spatial parameter in the original definition of $\left\{X_t\right\}$, it will be referred to as time from here on out in order to maintain consistency with the expression ``continuous-time birth-death process".  In addition, the elements in the process $\left\{X_t\right\}$ will be called particles, rather than frogs.  A particle will be said to ``die" at time $t_0$ if the point furthest to the right reached by that particle is $x=t_0$.
  
\subsection{ The process $\left\{Y_t\right\}$.} \label{ss:process Yt} The statement \eqref{Xtequiv}, along with the scale invariance of the original model with respect to $\lambda$ discussed in the introduction, together imply that Theorem \ref{theorem:neat3} can be proven by showing that the process $\left\{X_t\right\}$ goes extinct with probability 1 if and only if formula \eqref{int-cond} holds (for the case $\lambda=\frac{1}{2}$).  So let $\left\{Y_t\right\}$ be another continuous-time-inhomogeneous birth-death process with birth rate $f(t)$ and death rate $Y_t$ (hence, it differs from $\left\{X_t\right\}$ only in the sense that $0$ is not an absorbing state).  $\left\{Y_t\right\}$ is identified with a triple $(\Omega,\mathcal{F},{\bf P})$ defined as follows: $\Omega$ will represent the set of all functions $\omega :[0,\infty)\rightarrow\mathbb{N}$ such that $\omega(0)=1$ and where $\omega$ is constant everywhere except at a countable collection of points $p_1<p_2<\dots$, where for each $i\geq 1$, $\omega(p_i)=\lim\limits_{t\rightarrow p^-_i}\omega(t)\pm 1$ (note that $\Omega$ can be thought of as the collection of all possible paths $\left\{Y_t\right\}$ can take).  Let $\mathcal{F}$ denote the $\sigma$-field on $\Omega$ generated by the finite dimensional sets $\left\{\omega:\omega(s_i)=C_i\ \text{for }1\leq i\leq n\right\}$ (where $0<s_1<\dots <s_n$ and $C_i\in\mathbb{N}$ for each i).  Finally, ${\bf P}$ will refer to the probability measure on $(\Omega,\mathcal{F})$ associated with $\left\{Y_t\right\}$.  The primary task involved in moving towards a proof of Theorem \ref{theorem:neat3} will consist of proving a statement about $\left\{Y_t\right\}$.  This comes in the form of Theorem \ref{theorem:neat2}.
  
  \begin{theorem}\label{theorem:neat2}
  Assume $f:[0,\infty)\rightarrow [0,\infty)$ is monotonically increasing and define the value $V(\omega)=\#\left\{\text{points where }\omega\text{ jumps from }0\text{ to }1\right\}$.  Then ${\bf P}(V=\infty )=1$ if and only if\begin{equation}\label{int-inf}\int_0^{\infty}e^{-f(t)}f(t)dt=\infty\end{equation}If \eqref{int-inf} does not hold, then ${\bf P}(V=\infty)=0$.
  \end{theorem}
  
  \medskip
  \noindent The proof of Theorem \ref{theorem:neat2} has three main steps.  The first one entails proving the following proposition.  Note that in the statement of this proposition, and those following it, ${\bf E}$ will denote expectation with respect to the probability measure ${\bf P}$ and $f$ is assumed to be monotonically increasing.
  
  \begin{prop}\label{prop:neat}
  \begin{equation}\label{int-implies-exp}\int_0^{\infty}e^{-f(t)}f(t)dt=\infty\implies{\bf E}[V]=\infty\end{equation}
  \end{prop}
  
  \medskip
  \noindent After Proposition \ref{prop:neat} is established, it is then shown how the result can be used to prove one direction of Theorem \ref{theorem:neat2}.  This entails establishing the following implication.
  
\begin{prop}\label{prop:neat1}
${\bf E}[V]=\infty\implies{\bf P}(V=\infty)=1$.
\end{prop}

\medskip
\noindent After establishing Proposition \ref{prop:neat1}, we address the issue of proving the other direction of Theorem \ref{theorem:neat2}.  This comes in the form of the proceeding proposition.

\begin{prop}\label{prop:neat4}
\begin{equation}\int_0^{\infty}e^{-f(t)}f(t)dt<\infty\implies{\bf P}(V<\infty)=1\end{equation}
\end{prop}

\bigskip
\begin{proof}[Proof of Proposition \ref{prop:neat}]
  First note that for any $t>0$, $Y_t$ is a random variable of the form $\text{Bern}(e^{-t})+\text{Poiss}(\lambda_t)$ (with the two parts of the sum independent) where\begin{equation}\label{lambeq}\lambda_t=\int_0^tf(u)e^{-(t-u)}du=f(t)-f(t)e^{-t}-\int_0^t(f(t)-f(u))e^{-(t-u)}du\end{equation}  This follows from the fact that the single particle we began with at time zero remains ``alive" at time $t$ with probability $e^{-t}$ (hence the term $\text{Bern}(e^{-t})$), along with the fact that, if $f$ is continuous at $u$ (note that $f$ being increasing implies it is continuous a.e.), then the event of a particle being ``born" inside the time interval $[u, u+du)$ and surviving until at least time $t$, has probability $(f(u)e^{-(t-u)}+o(1))du$ as $du\rightarrow 0$ (where disjoint intervals are independent).  It is then implied by \eqref{lambeq} that $\lambda_t\leq f(t)\ \forall\ t\in [0,\infty)$, from which it follows that\begin{equation}\label{prob-asympt}{\bf P}(\omega(t)=0)\geq e^{-f(t)}(1-e^{-t})\end{equation}Since the probability $\left\{Y_t\right\}$ jumps from 0 to 1 on an interval $[t,t+dt)$ is $(1+o(1)){\bf P}(\omega(t)=0)f(t)dt$ as $dt\rightarrow 0$, it follows that\begin{equation}\label{expV}{\bf E}[V]=\int_0^{\infty}{\bf P}(\omega(t)=0)f(t)dt\end{equation}Combining this with \eqref{prob-asympt} then establishes the implication$$\int_0^{\infty}e^{-f(t)}f(t)dt=\infty\implies{\bf E}[V]=\infty$$Hence, the proof is complete.
\end{proof}

\begin{proof}[Proof of Proposition \ref{prop:neat1}]
 Let $U_x=\left\{\omega\in\Omega :\omega\text{ never jumps from }0\text{ to }1\text{ in }(x,\infty)\right\}$.  If it's assumed that ${\bf P}(V=\infty)<1$, then this implies that there exists $x,L>0$ (with $L\in\mathbb{Z}$) such that ${\bf P}(U_{x+1}|\ \omega(x+1)=L)>0$.  Since \begin{equation}\label{prob-pos}{\bf P}(\left\{\omega(t)>0\text{ on }(x,x+1)\right\}\cap\left\{\omega(x+1)=L\right\}|\ \omega(x)=1)>0\end{equation} we get ${\bf P}(U_x|\ \omega(x)=1)>0$.  Since $f$ is monotonically increasing we can couple $\big(\left\{Y_{t_1+t}\right\}|(Y_{t_1}=1)\big)$ with $\big(\left\{Y_{t_2+t}\right\}|(Y_{t_2}=1)\big)$ (for $t_2>t_1$) so that the former is dominated by the latter.  From this it follows that ${\bf P}(U_t|\ \omega(t)=1)$ is increasing w.r.t. $t$.  Now if we let $V_x(\omega)=\#\left\{\text{points in }(x,\infty)\text{ where }\omega\text{ jumps from }0\text{ to }1\right\}$, the fact that ${\bf P}(U_t|\omega(t)=1)$ is positive (for $t\geq x$) and increasing implies that ${\bf P}(V_x\geq T+1|\ V_x\geq T)\leq 1-{\bf P}(U_x|\ \omega(x)=1)\ \forall\ T\in\mathbb{N}$.  Hence,$${\bf E}[V_x]\leq\sum_{j=0}^{\infty}(1-{\bf P}(U_x|\ \omega(x)=1))^j=\frac{1}{{\bf P}(U_x|\ \omega(x)=1)}<\infty\implies{\bf E}[V]<\infty$$Therefore, we've established the contrapositive of Proposition \ref{prop:neat1}, which establishes the proposition.
\end{proof}

\begin{proof}[Proof of Proposition \ref{prop:neat4}]
It follows from \eqref{prob-asympt} and \eqref{expV} that\begin{equation}\label{expVF}{\bf E}[V]=\int_0^{\infty}e^{-\lambda_t}(1-e^{-t})f(t)dt\end{equation}Since $1-e^{-t}\rightarrow 1$ as $t\rightarrow\infty$, in order to show that ${\bf E}[V]<\infty$ it suffices to establish the following implication.\begin{equation}\label{boundimp}\int_0^{\infty}e^{-f(t)}f(t)dt<\infty\implies\int_0^{\infty}e^{-\lambda_t}f(t)dt<\infty\end{equation}Using the integral formula for $\lambda_t$ in \eqref{lambeq}, we see that if $f$ is continuous at $t$, then $\lambda_t$ is differentiable at $t$ with$$\frac{d\lambda_t}{dt}=\frac{d\Big(e^{-t}\int_0^t e^u f(u)du\Big)}{dt}=f(t)-e^{-t}\int_0^t e^u f(u)du=f(t)-\lambda_t$$Hence, at all continuity points of $f$, we have $f(t)=\lambda_t+\lambda'_t$.  Since $f$ is monotonically increasing, this means it has only countably many discontinuity points, which means it is continuous a.e (as was mentioned in the proof of \ref{prop:neat}).  It then follows that $f(t)=\lambda_t+\lambda'_t$ a.e.  Hence, we can write\begin{equation}\label{double-equ}\int_0^\infty e^{-f(t)}f(t)dt=\int_0^{\infty}e^{-(\lambda_t +\lambda'_t)}f(t)dt=\int_0^{\infty}e^{-\lambda_t}f(t)e^{-\lambda'_t}dt\end{equation}Now let $0\leq t_1<t_2$ and note that\begin{equation}\label{monoforl}\lambda_{t_2}-\lambda_{t_1}=\int_0^{t_2-t_1}e^{-(t_2-u)}f(u)du+\int_0^{t_1}e^{-(t_1-u)}\Big(f(t_2-t_1+u)-f(u)\Big)du>0\end{equation}Hence, $\lambda_t$ is monotonically increasing.  Also note that if $0\leq t_1<t_2\leq N$ (for $N<\infty$) then$$\lambda_{t_2}=\int_0^{t_2}e^{-(t_2-u)}f(u)du=e^{(t_1-t_2)}\int_0^{t_1}e^{-(t_1-u)}f(u)du+\int_{t_1}^{t_2}e^{-(t_2-u)}f(u)du\leq\lambda_{t_1}+(t_2-t_1)f(N)$$Along with \eqref{monoforl} this implies $|\lambda_{t_2}-\lambda_{t_1}|\leq (t_2-t_1)f(N)$, which means that $\lambda_t$ is absolutely continuous on $[0,N]$.  Coupled with $\lambda_t$ being monotonically increasing and satisfying $\lambda_0=0$, this implies that if $g:[0,N]\rightarrow [0,\infty)$ is any Lebesgue measurable function, then\begin{equation}\label{impint}\int_0^{\lambda_N}g(x)dx=\int_0^N g(\lambda_t)\lambda'_t dt\end{equation}(see \cite{WR}, para. 2, pg. 156).  Now since $\underset{t\rightarrow\infty}{\text{lim}}f(t)=\infty$ (otherwise it could not hold that $\int_0^{\infty}e^{-f(t)}f(t)dt<\infty$) this means $\lambda_t=\int_0^t e^{-(t-u)}f(u)du\rightarrow\infty$ as $t\rightarrow\infty$.  Therefore, if $g\geq 0$ is Lebesgue measurable with $\int_0^{\infty}g(x)dx<\infty$, then letting $N\rightarrow\infty$ in \eqref{impint} gives$$\int_0^{\infty}g(x)dx=\int_0^{\infty}g(\lambda_t)\lambda'_t dt$$Specifically looking at the cases $g_1(x)=e^{-x}$ and $g_2(x)=xe^{-x}$, gives the two formulas$$\int_0^{\infty}e^{-\lambda_t}\lambda'_tdt=\int_0^{\infty}e^{-x}dx=1$$ $$\int_0^{\infty}e^{-\lambda_t}\lambda_t\lambda'_tdt=\int_0^{\infty}xe^{-x}dx=1$$Combining these formulas with \eqref{double-equ}, and using the fact that $f(t)=\lambda_t+\lambda'_t$ a.e., then gives\begin{eqnarray*}
\int_0^{\infty}e^{-\lambda_t}f(t)dt & = & \int_0^{\infty}e^{-\lambda_t}f(t)1_{(\lambda'_t\leq 1)}dt+\int_0^{\infty}e^{-\lambda_t}(\lambda_t+\lambda'_t)1_{(\lambda'_t>1)}dt \\ 
& \leq & e\int_0^{\infty}e^{-\lambda_t}f(t)e^{-\lambda'_t}dt+\int_0^{\infty}e^{-\lambda_t}\lambda_t\lambda'_tdt\ +1=e\int_0^{\infty}e^{-f(t)}f(t)dt\ +2<\infty
\end{eqnarray*}Hence, this establishes \eqref{boundimp} which, as was shown, implies ${\bf E}[V]<\infty$, from which it follows that ${\bf P}(V<\infty)=1$.  Hence, the proof is complete.
\end{proof}

\begin{proof}[Proof of Theorem \ref{theorem:neat2}]
The theorem follows immediately from Propositions \ref{prop:neat}, \ref{prop:neat1}, and \ref{prop:neat4}.
\end{proof}

\subsection{ Proving Theorem \ref{theorem:neat3}} \label{ss:pt11} With Theorem \ref{theorem:neat2} established, we can now complete the proof of Theorem \ref{theorem:neat3}.  Due to the relationship discussed earlier between the process $\left\{X_t\right\}$ and the frog model with drift on $\mathbb{R}$, as well as scale invariance of the original model, it suffices to prove the following claim.

\bigskip
\noindent {\bf Claim:} For $f:[0,\infty)\rightarrow [0,\infty)$ monotonically increasing, the process $\left\{X_t\right\}$ dies out with probability 1 if and only if$$\int_0^{\infty}e^{-f(t)}f(t)dt=\infty$$

\begin{proof}
First couple the process $\left\{X_t\right\}$ with the familiar process $\left\{Y_t\right\}$ so that the two processes are identical until $\left\{X_t\right\}$ dies out.  Since Theorem \ref{theorem:neat2} established the implication\begin{equation}\label{jumps-from}\int_0^{\infty}e^{-f(t)}f(t)dt=\infty\implies{\bf P}(\left\{Y_t\right\}\text{ jumps from 0 to 1 i.o.})=1\end{equation}this means that if the left side of \eqref{jumps-from} holds, then with probability 1, $\left\{X_t\right\}$ will eventually die out (i.e. the coupled processes $\left\{Y_t\right\}$ and $\left\{X_t\right\}$ will eventually hit 0).  Conversely, since Theorem \ref{theorem:neat2} also states that if the integral in \eqref{jumps-from} is finite then $V<\infty$ with probability 1, this means that if we let $T_0(\omega)=\left\{t\in[0,\infty):\omega(t)=0\right\}$, then ${\bf P}(\text{sup}\ T_0<\infty)=1$.  Letting $\Prob$ represent the law of $\left\{X_t\right\}$ on $(\Omega, \mathcal{F})$, it now follows that there must be some $t>0$ and some positive integer $M$ s.t.\begin{equation}\label{prob-eq-prob}\Prob(X_s>0\ \forall\ s\geq t|X_t=M)={\bf P}(Y_s>0\ \forall\ s\geq t|Y_t=M)>0\end{equation}and ${\bf P}(Y_t=M)>0$.  This then implies that ${\bf P}(X_t=M)>0$, which along with \eqref{prob-eq-prob}, gives$$\int_0^{\infty}e^{-f(t)}f(t)dt<\infty\implies\Prob(X_s>0\ \forall\ s\in [0,\infty))\geq\Prob(X_t=M)\Prob(X_s>0\ \forall\ s\geq t|X_t=M)>0$$Alongside the first part of the proof, this last result establishes that $\left\{X_t\right\}$ dies out with probability 1 if and only if the integral on the left side of \eqref{jumps-from} diverges.  Thus we have established the above claim, which completes the proof of Theorem \ref{theorem:neat3}.
\end{proof}

\medskip
\noindent $Remark\ 4.$ Note that the result of Theorem \ref{theorem:neat3} can easily be extended to all measurable functions $f:[0,\infty)\rightarrow[0,\infty)$ for which $\exists\ r\in (0,\infty)$ such that $f$ is bounded on $[0,r)$ and increasing on $[r,\infty)$.  This is established by first noting that it follows from \eqref{Xtequiv} that the process is non-transient if and only if\begin{equation}\label{remequiv}\mathbb{P}(\underset{t\rightarrow\infty}{\text{lim}}X_t>0)>0\Longleftrightarrow\mathbb{P}(X_r>0)\mathbb{P}(\underset{t\rightarrow\infty}{\text{lim}}X_t>0|X_r>0)>0\end{equation}Because $\mathbb{P}(X_r>0)\geq e^{-r}$ (since the particle beginning at time $0$ remains ``alive" at time $r$ with probability $e^{-r}$), and because $\mathbb{P}(\underset{t\rightarrow\infty}{\text{lim}}X_t>0|X_r=L)>0$ (for some $L>0$) if and only if $\mathbb{P}(\underset{t\rightarrow\infty}{\text{lim}}X_t>0|X_r=1)>0$, it follows from \eqref{remequiv} and Theorem \ref{theorem:neat3} that the process is non-transient if and only if$$\mathbb{P}(\underset{t\rightarrow\infty}{\text{lim}}X_t>0|X_r=1)>0\Longleftrightarrow\int_r^{\infty}e^{-f(t)}f(t)dt<\infty\Longleftrightarrow\int_0^{\infty}e^{-f(t)}f(t)dt<\infty$$

\section{Transience vs. non-transience for the model on $\mathbb{Z}$}

\subsection{The processes $\left\{M_j\right\}$ and $\left\{N_j\right\}$} Take the non-uniform frog model with drift on $\mathbb{Z}$ and define the process $\left\{M_j\right\}$ as follows.  Let $M_0 =1$ and, for $j\geq 1$, let $M_j$ equal the number of frogs originating in $\left\{0,1,\dots,j-1\right\}$ that ever hit $x=j$.  Much like with the process $\{X_t\}$, we find that\begin{equation}\label{inf-frogs}\lim\limits_{j\rightarrow\infty}M_j >0\Longleftrightarrow\{\text{infinitely many frogs return to the origin}\}\end{equation}Examining the process $\{M_j\}$, we also obtain this next proposition.

\begin{prop}\label{prop:neat6}
$\{M_j\}$ is a discrete-time-inhomogeneous Markov process with $M_0 =1$, $M_1=\text{Bern}(\frac{1-p}{p})$, and for $j\geq 1$ $$(M_{j+1}|M_j)=\left\{\begin{array}{ll}\text{Bin}\big(M_j,\frac{1-p}{p}\big)+\text{Poiss}\big(\frac{1-p}{p}f(j)\big) &\text{if }M_j\geq 1\\ 0 &\text{if }M_j=0\end{array}\right.$$(where the two parts of the above sum are independent).
\end{prop}

\begin{proof}
    By a simple martingale argument the probability an active frog residing at $x=j$ ever makes it to $x=j+1$ is $\frac{1-p}{p}$.  Therefore, the expression for $M_1$ follows.  This also implies that if we condition on $M_j$, then for $j\geq 1$ the distribution of the number of frogs that make it to $x=j+1$ which originate in $\left\{0,1,\dots, j-1\right\}$, is $\text{Bin}\big(M_j,\frac{1-p}{p}\big)$.  Adding this to the number of frogs originating at $x=j$ that ever make it to $x=j+1$, while again using the first line of this proof along with the fact that $\eta_j$ (the number of sleeping frogs starting at $x=j$) has distribution $\text{Poiss}\big(f(j)\big)$, gives us the above piecewise expression for $\big(M_{j+1}|M_j\big)$.
\end{proof}
As stated earlier, $\left\{N_j\right\}$ will represent a process identical to $\left\{M_j\right\}$ except that $(N_{j+1}|N_j=0)$ has distribution $\text{Poiss}(\frac{1-p}{p}f(j))$.  $\left\{N_j\right\}$ is identified with a triple $(\Omega^*,\mathcal{F}^*,{\bf P}^*)$ defined as follows.  $\Omega^*$ will represent the set of all functions $\omega:\mathbb{N}\rightarrow\mathbb{N}$, $\mathcal{F}^*$ will represent the $\sigma\text{-field}$ on $\Omega^*$ generated by the finite dimensional sets, and ${\bf P}^*$ will refer to the probability measure on $(\Omega^*,\mathcal{F}^*)$ associated with $\left\{N_j\right\}$ (note that ${\bf P}^*$ is supported on $\left\{\omega\in\Omega^*:\omega(0)=1,\omega(1)\leq 1\right\}$).  Theorem \ref{theorem:neat7} can now be stated formally.

\begin{theorem}\label{theorem:neat7}
If $\frac{1}{2}<p<1$, $f$ is monotonically increasing, and we let $K(\omega)=\#\left\{j\in\mathbb{Z}^+:\omega(j)=0\right\}$, then ${\bf P}^*(K=\infty)=1$ if and only if\begin{equation}\label{sumeqinf}\sum_{j=1}^{\infty}e^{-\frac{1-p}{2p-1}f(j)}=\infty\end{equation}If \eqref{sumeqinf} does not hold then ${\bf P}^*(K=\infty)=0$.
\end{theorem}

\subsection{Proving Theorem \ref{theorem:neat5}} We begin this section by presenting Propositions \ref{prop:neat8}-\ref{prop:neat10}.  In places where the proofs bear an especially strong resemblance to those for the model on $\R$, some details are omitted.  In what follows, $f$ is always assumed to be monotonically increasing, and ${\bf E}^*$ will represent expectation with respect to ${\bf P}^*$.

\begin{prop}\label{prop:neat8}
$$\sum_{j=1}^{\infty}e^{-\frac{1-p}{2p-1}f(j)}=\infty\implies{\bf E}^*[K]=\infty$$
\end{prop}

\begin{proof}
As a random variable (for $j\geq 1$)\begin{equation}\label{bern}N_j=\text{Bern}\Bigg(\Big(\frac{1-p}{p}\Big)^j\Bigg)+\text{Poiss}(\tau_j)\end{equation}where$$\tau_j=\sum_{i=1}^{j-1}\Big(\frac{1-p}{p}\Big)^{j-i}f(i)$$By an argument similar to the one employed in the proof of Proposition \ref{prop:neat}, we find that it follows from the fact $f$ is increasing that $\tau_j\leq\frac{1-p}{2p-1}f(j)\ \forall\ j$.  Combining this with \eqref{bern} establishes that\begin{equation}\label{prob-zero-eq}{\bf P}^*(\omega(j)=0)\geq \Big(1-\Big(\frac{1-p}{p}\Big)^j\Big)e^{-\frac{1-p}{2p-1}f(j)}\end{equation}Since $\Big(\frac{1-p}{p}\Big)^j\rightarrow 0$ as $j\rightarrow\infty$ and$${\bf E}^*[K]=\sum_{j=1}^{\infty}{\bf P}^*(\omega(j)=0)$$the proposition follows.
\end{proof}

\begin{prop}\label{prop:neat9}
$${\bf E}^*[K]=\infty\implies{\bf P}^*(K=\infty)=1$$
\end{prop}

\begin{proof}
Proceed by proving the contrapositive.  Let $U_j=\left\{\omega\in\Omega^*:\omega(i)>0\ \forall\ i>j\right\}$.  Assume ${\bf P}^*(K=\infty)<1$.  It will follow that $\exists\ L\geq 1$ such that ${\bf P}^*(U_L|\omega(L)=0)>0$.  Along with the fact that ${\bf P^*}(U_L|\omega(L)=0)$ is monotonically increasing (which follows from the fact that $f$ is increasing), this implies that ${\bf E}^*[K]-L$ can be bounded above by the sum of a geometric series with base $1-{\bf P}^*(U_L|\omega(L)=0)<1$.  The contrapositive of the proposition then follows, which establishes the proposition.
\end{proof}

\begin{prop}\label{prop:neat10}
$$\sum_{j=1}^{\infty}e^{-\frac{1-p}{2p-1}f(j)}<\infty\implies{\bf P}^*(K<\infty)=1$$
\end{prop}

\begin{proof}
Since \eqref{bern} implies that ${\bf P^*}(\omega(j)=0)=e^{-\tau_j}\big(1-\big(\frac{1-p}{p}\big)^j\big)$, it follows that$${\bf E^*}[K]=\sum_{j=1}^{\infty}e^{-\tau_j}\Big(1-\Big(\frac{1-p}{p}\Big)^j\Big)<\sum_{j=1}^{\infty}e^{-\tau_j}$$Hence, to show that$$\sum_{j=1}^{\infty}e^{-\frac{1-p}{2p-1}f(j)}<\infty\implies{\bf E^*}[K]<\infty$$it suffices to show that\begin{equation}\label{sumimp1}\sum_{j=1}^{\infty}e^{-\frac{1-p}{2p-1}f(j)}<\infty\implies\sum_{j=1}^{\infty}e^{-\tau_j}<\infty\end{equation}To establish \eqref{sumimp1}, first note that
\begin{eqnarray*}
\tau_{j+1}-\tau_j & = & \Big(\frac{1-p}{p}\Big)^{j+1}\sum_{i=1}^j \Big(\frac{1-p}{p}\Big)^{-i}f(i)-\Big(\frac{1-p}{p}\Big)^j\sum_{i=1}^{j-1}\Big(\frac{1-p}{p}\Big)^{-i}f(i)\\ & = & \Big[\Big(\frac{1-p}{p}\Big)^{j+1}-\Big(\frac{1-p}{p}\Big)^j\Big]\sum_{i=1}^{j-1}\Big(\frac{1-p}{p}\Big)^{-i}f(i) +\frac{1-p}{p}f(j)\\ & = & \frac{1-2p}{p}\tau_j+\frac{1-p}{p}f(j)\implies\frac{1-p}{2p-1}f(j)=\frac{p}{2p-1}\Delta\tau_j+\tau_j
\end{eqnarray*}(where $\Delta\tau_j$ denotes $\tau_{j+1}-\tau_j$).  Hence,\begin{equation}\label{sumequ2}\sum_{j=1}^{\infty}e^{-\frac{1-p}{2p-1}f(j)}=\sum_{j=1}^{\infty}e^{-\tau_j}\cdot e^{-\frac{p}{2p-1}\Delta\tau_j}\end{equation}Now if the sum on the right in \eqref{sumimp1} is written as\begin{equation}\label{decomp1}\sum_{j=1}^{\infty}e^{-\tau_j}=\sum_{j=1}^{\infty}e^{-\tau_j}1_{(\Delta\tau_j\leq 1)}+\sum_{j=1}^{\infty}e^{-\tau_j}1_{(\Delta\tau_j>1)}\end{equation}then \eqref{sumequ2} and the left side of \eqref{sumimp1} imply that$$\sum_{j=1}^{\infty}e^{-\tau_j}\leq e^{\frac{p}{2p-1}}\sum_{j=1}^{\infty}e^{-\frac{1-p}{2p-1}f(j)}+\frac{e}{e-1}<\infty$$(where the $\frac{e}{e-1}$ term follows from the fact that the last sum on the right in \eqref{decomp1} can be bounded above by the sum of the geometric series with base $e^{-1}$).  Therefore, this establishes \eqref{sumimp1}, which implies ${\bf E^*}[K]<\infty$, from which it follows that ${\bf P^*}(V<\infty)=1$.  Hence, the proof is complete.
\end{proof}

\begin{proof}[Proof of Theorem \ref{theorem:neat7}]
The theorem is an immediate consequence of Propositions \ref{prop:neat8}-\ref{prop:neat10}.
\end{proof}

Theorem \ref{theorem:neat7} is now used to establish Theorem \ref{theorem:neat5}.  Note that on account of \eqref{inf-frogs}, establishing Theorem \ref{theorem:neat5} reduces to proving the following claim.

\bigskip
\noindent {\bf Claim:} For $f:\mathbb{Z}^+\rightarrow [0,\infty)$ monotonically increasing, the process $\left\{M_j\right\}$ dies out with probability 1 if and only if$$\sum_{j=1}^{\infty}e^{-\frac{1-p}{2p-1}f(j)}=\infty$$

\begin{proof}
By a coupling of $\left\{M_j\right\}$ with $\left\{N_j\right\}$, it follows from Theorem \ref{theorem:neat7} that if the above sum diverges, then $\left\{M_j\right\}$ dies out with probability 1.  For the other direction, we can apply an argument exactly like the one we used in the continuous case, but where we replace the integral with the sum and replace $\left\{X_t\right\}$ and $\left\{Y_t\right\}$ with $\left\{M_j\right\}$ and $\left\{N_j\right\}$ respectively.  Alongside the first part of the proof, this establishes Theorem \ref{theorem:neat5}.
\end{proof}

\medskip
\noindent $Remark\ 5.$ Much like with the continuous case, the result of Theorem \ref{theorem:neat5} extends to all functions $f:\mathbb{Z}^+\rightarrow[0,\infty)$ for which $\exists\ q\in\mathbb{Z}^+$ such that $f$ is increasing on $\left\{q,q+1,q+2,\dots\right\}$.  Due to its strong similarity to the argument given in $Remark\ 4$, the explanation for this is omitted.

\section{Counterexamples and additional comments}

\subsection{Counterexamples} In this final section I'll discuss a scenario in which $f:[0,\infty)\rightarrow[0,\infty)$ is not monotonically increasing, and the tight condition of Theorem \ref{theorem:neat3} ceases to hold.  A similar case for the discrete model is also mentioned.

\medskip
\begin{example}\label{example:2ndex}
Define $f:[0,\infty)\rightarrow [0,\infty)$ as$$f(t)=\left\{\begin{array}{ll}1 &\text{if }t\in[2^n,2^n +1)\ \text{for }n\in\mathbb{Z}^+\\ t &\text{otherwise}\end{array}\right.$$Since ${\bf E}[V]=\int_0^{\infty}e^{-\lambda_t}f(t)(1-e^{-t})dt$ (see \eqref{expVF}), to show that ${\bf E}[V]<\infty$ it suffices to show that $\int_0^{\infty}e^{-\lambda_t}f(t)dt<\infty$.  Recalling from \eqref{lambeq} that $\lambda_t=\int_0^t e^{-(t-u)}f(u)du$, we'll seek to achieve a lower bound for $\lambda_t$.  
Note first that if $n\in\mathbb{Z}^+$ then$$\int_{2^n+1}^{2^{n+1}}e^u f(u)du=\int_{2^n+1}^{2^{n+1}}ue^u du=\big(2^{n+1}-1\big)e^{2^{n+1}}-2^n\cdot e^{2^n+1}$$Hence, for $t=2^{n+1}$ (for $n\in\mathbb{Z}^+$)\begin{equation}\label{equstring}\lambda_t=e^{-t}\Big(\int_0^2 ue^u du+\sum_{j=1}^n\int_{2^j}^{2^j+1}e^u du+\sum_{j=1}^n\big(2^{j+1}-1\big)e^{2^{j+1}}-2^j\cdot e^{2^j+1}\Big)\geq e^{-t}\Big(\big(t-1\big)e^t-\frac{t}{2}e^{\frac{t}{2}+1}\Big)\geq t-2\end{equation}(where the last inequality follows from the fact that $\frac{t}{2}e^{\frac{t}{2}+1}<e^t$ for $t\geq 4$).  
Since $\lambda_{t+r}\geq e^{-r}\lambda_t$, it follows that for $0\leq r\leq 1$ (with $t=2^{n+1}$ as above) we have\begin{equation}\label{lambLB}\lambda_{t+r}\geq e^{-1}(t-2)\end{equation}Furthermore, note that if $t_o\in(t+1,2t)$ then $\lambda'_{t_o}$ exists (since $f$ is continuous in $(2^{n+1}+1,2^{n+2})$) with $\lambda'_{t_o}=f(t_o)-\lambda_{t_o}$.  
Since $\lambda_{t_o}\leq e^{-t_o}\int_0^{t_o}u e^u du=t_o-1+e^{-t_o}$, this means$$\lambda_{t_o}\leq t_o-e^{-1}\implies\lambda'_{t_o}=t_o-\lambda_{t_o}\geq e^{-1}$$ which along with \eqref{lambLB}, implies $\lambda_{t_o}\geq e^{-1}(t_o-2)$.  Combining this with \eqref{equstring} and \eqref{lambLB} then tells us that $\lambda_s\geq e^{-1}(s-2)\ \forall\ s\in [2^{n+1},2^{n+2})$ (for $n\in\mathbb{Z}^+$), and therefore $\lambda_s\geq e^{-1}(s-2)\ \forall\ s\in [4,\infty)$.  Using this inequality, along with the fact that $f(s)\leq s\ \forall\ s\in [0,\infty)$, we find that$$\int_0^{\infty}e^{-\lambda_t}f(t)dt\leq\int_0^4 t dt+\int_4^{\infty}e^{-e^{-1}(t-2)}t dt<\infty\implies{\bf E}[V]<\infty\implies{\bf P}(V<\infty)=1$$By the argument that was employed in Section 2.3 to establish the implication ${\bf P}(V<\infty)=1\implies\left\{\text{non-transience}\right\}$, it follows that for the given Poisson intensity function $f$ (with drift $\frac{1}{2}$) the model is non-transient.  Noting now that $$\int_0^{\infty}e^{-f(t)}f(t)dt\geq\sum_{j=1}^{\infty}\int_{2^j}^{2^j+1}e^{-1}dt=\sum_{j=1}^{\infty}e^{-1}=\infty$$we find that the tight condition from Theorem \ref{theorem:neat3} does indeed fail to apply in this case.
\end{example}

\medskip
\noindent $Remark\ 6.$ Notice that the tight condition of Theorem \ref{theorem:neat3} also fails to hold when $f:[0,\infty)\rightarrow[0,\infty)$ is a bounded (nonzero) function such that $\int_0^{\infty}f(x)dx<\infty$ (since then $\int_0^{\infty}e^{-f(x)}f(x)dx<\infty$, but the model is transient).  However, if the integral in \eqref{int-cond} is changed to $\int_0^{\infty}e^{-f(x)}(1+f(x))dx$, then \ref{theorem:neat3} remains valid, yet functions in $L^1([0,\infty))$ that are bounded, nonzero, and nonnegative, no longer violate the new condition.  Hence, such functions offer far less insight into the limits to which the result of Theorem \ref{theorem:neat3} can be stretched, than does the case examined in Example \ref{example:2ndex}.

\medskip
\begin{example}\label{example:3rdex}
Define $f:\mathbb{Z}^+\rightarrow\mathbb{N}$ as$$f(j)=\left\{\begin{array}{ll}1 &\text{if }j=2^n\ \text{for }n\in\mathbb{Z}^+\\ j &\text{otherwise}\end{array}\right.$$It follows from \eqref{bern} that in order to show that ${\bf E^*}[K]<\infty$ it suffices to show that $\sum_{j=1}^{\infty}e^{-\tau_j}<\infty$ (with $\tau_j$ defined as in Section 3.2).  From the formulas for $\tau_j$ and $f$ we see that $\tau_j\geq\Big(\frac{1-p}{p}\Big)^2(j-2)\ \forall\ j\geq 1$ (recall $\frac{1}{2}<p<1$).  Hence, $$\sum_{j=1}^{\infty}e^{-\tau_j}<\infty\implies{\bf E^*}[K]<\infty\implies{\bf P^*}(K<\infty)=1$$As we saw in the proofs of Theorems \ref{theorem:neat3} and \ref{theorem:neat5}, this implies non-transience of the model.  Combining this with the fact that$$\sum_{j=1}^{\infty}e^{-\frac{1-p}{2p-1}f(j)}\geq\sum_{j=1}^{\infty}e^{-\frac{1-p}{2p-1}f(2^j)}=\sum_{j=1}^{\infty}e^{-\frac{1-p}{2p-1}}=\infty$$we see that the tight condition of Theorem \ref{theorem:neat5} does not apply in this case.
\end{example}

\section*{Acknowledgements}

The author would like to thank Toby Johnson for providing extensive background on the frog model; thanks also to Marcus Michelen for helpful conversations and technical assistance.


\begin{thebibliography}{9}

\bibitem{BMZ}
Daniela Bertacchi, Fabio Prates Machado, and Fabio Zucca, \textit{Local and global survival for nonhomogeneous random walk systems on} $\mathbb{Z}$, Advances in Applied Probability, \textbf{46} (2012), no. 1, 256-278. MR 3189058

\bibitem{DP}
Christian Dobler and Lorenz Pfeifroth, \textit{Recurrence for the frog model with drift on} $\mathbb{Z}^d$, Electron. Commun. Probab. \textbf{19} (2014), no. 79, 13. MR 3283610

\bibitem{Gantert}
N. Gantert and P. Schmidt, \textit{Recurrence for the frog model with drift on }$\mathbb{Z}$, Markov Process. Related Fields \textbf{15} (2009), no. 1, 51-58. MR 2509423

\bibitem{HJJ1}
Christopher Hoffman, Tobias Johnson, and Matthew Junge, \textit{From transience to recurrence with Poisson tree frogs}, available at arXiv:1501.05874, 2015

\bibitem{HJJ2}
Christopher Hoffman, Tobias Johnson, and Matthew Junge, \textit{Recurrence and transience for the frog model on trees}, available at arXiv:1404.6238, 2015.

\bibitem{WR}
Walter Rudin, \textit{Real and Complex Analysis}, 3rd ed., McGraw-Hill, Inc., New York, NY, 1987.

\bibitem{TW}
Andr\'as Telcs and Nicholas C. Wormald, \textit{Branching and tree indexed random walks on fractals}, J. Appl. Probab. \textbf{36} (1999), no. 4, 999-1011. MR 1742145

\end{thebibliography}
\end{document}